\documentclass[a4paper,12pt]{article}
\usepackage[top=3.0cm,bottom=3.0cm,left=2.5cm,right=2.5cm]{geometry}
\usepackage{cite, amsmath, amssymb, amsthm}
\usepackage[margin=1cm,%
font=small,%
format=hang,%
labelsep=period,%
labelfont=bf]{caption}
\usepackage{multirow}
\usepackage{tkz-graph}
\usepackage{tikz}
\usepackage{enumerate}
\usepackage{enumitem}
\usepackage{lscape}
\usepackage{comment}

\usepackage{multirow}
\usepackage{tkz-graph}
\usepackage{tikz}
\usepackage[ruled, vlined]{algorithm2e}
\usepackage{enumerate}
\usepackage{lscape}

\newtheorem{theorem}{Theorem}[section]

\newtheorem{conjecture}[theorem]{Conjecture}

\newtheorem{lemma}[theorem]{Lemma}

\newtheorem{remark}[theorem]{Remark}

\begin{document}

\vspace*{6cm}

\begin{center}
{\LARGE \textbf{On the Graovac-Ghorbani index for bicyclic graphs with no pendant vertices}}

\vspace{1cm}

{\large \textbf{Diego Pacheco$^1$, Leonardo de Lima$^2$, Carla Silva Oliveira$^3$}}\\

\vspace{0.5cm}

\emph{$^1$Departamento de Engenharia de Produ\c c\~ao, Centro Federal de Educa\c c\~ao Tecnol\'ogica Celso Suckow da Fonseca, Rio de Janeiro, Brazil}\\ 

\vspace{0.2cm}

\emph{$^2$ Departamento de Administra\c c\~ao Geral e Aplicada -Universidade Federal do Paran\'a, Av. Prefeito Loth\'ario Meissner -$2^{o}$andar, 80210-170, Curitiba, PR, Brazil}\\

\vspace{0.2cm}

\emph{$^3$Escola Nacional de Ciências Estat\'{\i}sticas, Rua Andr\'e Cavalcanti 106, Bairro de F\'atima, 20231-050, RJ, Brazil}\\ \vspace{0.2cm}  

diego.pacheco@aluno.cefet-rj.br, leonardo.delima@ufpr.br, carla.oliveira@ibge.gov.br
\vspace{0.4cm}

\end{center}

\begin{abstract}

Let $G=(V,E)$  be a simple undirected and  connected graph on $n$ vertices. The Graovac--Ghorbani index of a graph $G$ is defined as 
$$ABC_{GG}(G)= \sum_{uv \in E(G)} \sqrt{\frac{n_{u}+n_{v}-2} {n_{u} n_{v}}},$$
where $n_u$ is the number of vertices closer to vertex $u$ than vertex $v$ of the edge $uv \in E(G)$ and $n_{v}$ is defined analogously. It is well-known that all bicyclic graphs with no pendant vertices are composed by three families of graphs, which we denote by $\mathcal{B}_{n} = B_1(n) \cup B_2(n) \cup B_3(n).$ In this paper, we give an lower bound to the $ABC_{GG}$ index for all graphs in $B_1(n)$ and prove it is sharp by presenting its  extremal graphs. Additionally, we  conjecture a sharp lower bound to the $ABC_{GG}$ index for  all graphs in $\mathcal{B}_{n}.$

\medskip



\end{abstract}

\baselineskip=0.30in

\section{Introduction}

Let $G=(V,E)$  be a simple undirected and connected graph such that $n=|V|$ and $m=|E|$.  The degree of a vertex $v \in V$, denoted by $d_{v}$, is the number of edges incidents to $v$. The Graovac-Ghorbani index, \cite{Graovac2010}, is defined as  
\begin{equation}\label{eq:gg}
ABC(G)= \sum_{uv \in E(G)} \sqrt{\frac{n_{u}+n_{v}-2} {n_{u} n_{v}}},    
\end{equation}
where $n_u$ is the number of vertices closer to vertex $u$ than vertex $v$ of the edge $uv \in E(G)$ and $n_{v}$ is defined analogously. Note that equidistant vertices from $u$ and $v$ are not taken into account to compute $n_u$ and $n_v$ in  Equation (\ref{eq:gg}). The problem of finding graphs with maximum or minimum Graovac-Ghorbani index  turns to be a difficult problem for general graphs.  Some papers has been published in order to find extremal graphs to the Graovac-Ghorbani index of trees \cite{Rostami2013}, unicyclic \cite{Das2013} and bipartite graphs \cite{Dimitrov2017}. Some interesting results on this topic can be found at  \cite{Rostami2014,Das2016b,Furtula2016,Das2016,Dimitrov2017b}.
In 2016, Das in \cite{Das2016} posed the following question: ``Which graph has minimal $ABC_{GG}$ index among all bicyclic graphs?''  Motivated by this question we considered the $ABC_{GG}$ index for bicyclic graphs with no pendant vertices. 

In this paper, we explicitly give the $ABC_{GG}$ index for some bicyclic graphs of no pendant vertices and present a sharp lower bound. Also, we conjecture a lower bound to the $ABC_{GG}$ for all bicyclic graphs with no pendant vertices.

\section{Preliminaries}

A connected graph $G$ of order $n$ is called a bicyclic graph if $G$ has $n+1$ edges. Bicyclic graphs with no vertex of degree one are bicyclic graphs with no pendant vertices. Let $\mathcal{B}_{n}$ be the set of all bicyclic graphs of order $n$ with no pendant vertices. It is well-known, \cite{He2007}, that there are three types of bicyclic graphs containing no pendant vertices, which we denote here by $B_1(n), B_2(n)$ and $B_3(n).$ We use integers $p, \, q \geq 3$ to denote the size of the cycles, and $l\geq 1$ to denote the length of a path (i.e., the number of edges of a path).   Let $B_{1}(p,q)$ be the set of bicyclic graphs obtained from two vertex-disjoint cycles $C_p$ and $C_q$ by identifying a vertex $u$ of $C_p$ and a vertex $v$ of $C_q$ such that $n = p+q-1.$ Observe that all graphs in $B_{1}(p,q)$ have the same number of vertices but are not isomomorphic since the size of the cycles are not the same. Let $B_{2}(p,l,q)$ be the bicyclic graph obtained from two vertex-disjoint cycles $C_p$ e $C_q$, by joining vertices $v_1$ of $C_p$ and $u_l$ of $C_q$ by a new path  $v_{1}u_1,\ldots,u_{l-1}u_{l}$ with length $l$, such that  $n = p+q+l-1.$ Let $B_{3}(p,l,q)$ be the bicyclic graph obtained from a cycle $C_{p+q-2l}$ with vertex set given by $v_1v_2v_3,\ldots,v_{p+q-2l-1}v_{p+q-2l}v_{1}$ by joining vertices $v_1$ and $v_{p-l-2}$ by a new path $v_1u_1u_2,\ldots, u_{l-2}u_{l-1}u_{l}v_{p-l-2}$ with length $l$, where $n = p+q-l-1.$ Thus, $$B_1(n) =  \bigcup_{p, q \geq 3}  B_1(p,q), \; B_2(n) =  \bigcup_{p, q \geq 3, l \geq 1}  B_2(p,l,q) \mbox{  and  }  B_3(n) =  \bigcup_{p, q \geq 3, l \geq 1}  B_3(p,l,q).$$

Now, it is clear that  $\mathcal{B}_n = B_1(n) \cup B_2(n) \cup B_3(n).$ In Figure \ref{fig1} the general form of the graphs in  families $B_1(n), B_2(n)$ and $B_3(n)$ is displayed.

\begin{figure}[!h]
    \centering
    \includegraphics[height=8cm]{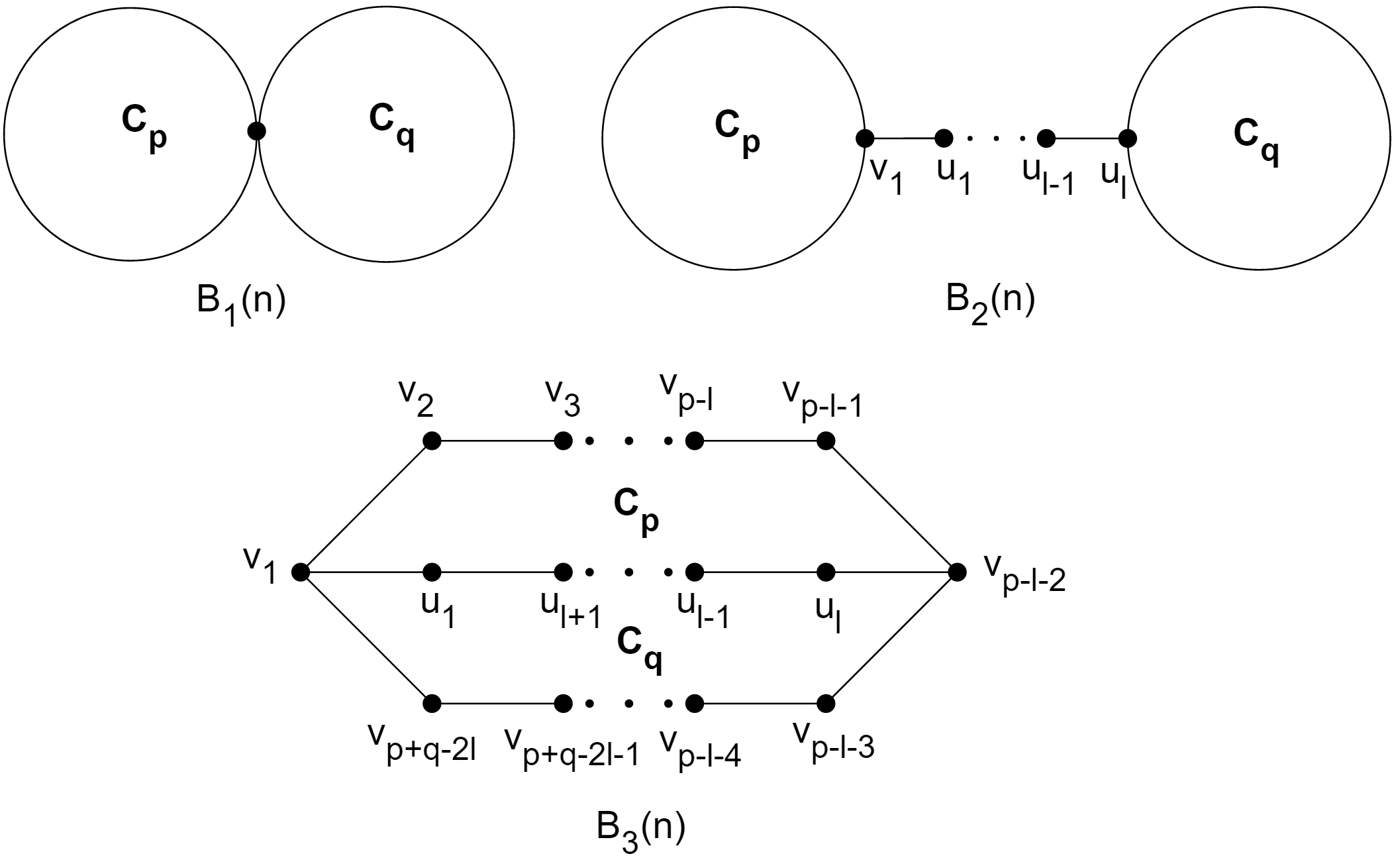}
    \caption{Families $B_1(n)$, $B_2(n)$ and $B_3(n)$ of bicyclic graphs with no pendant vertices}
    \label{fig1}
\end{figure}

\section{$ABC_{GG}$ index for  all graphs $G \in B_{1}(n)$}

In this section, we give an explicit formula to the $ABC_{GG}$ index of any graph in $B_1(n).$ In order to prove it, we consider the following cases:
\begin{itemize}
\item If $n$ is odd there are two possibilities: either $C_p$ and $C_q$ are both odd cycles or $C_p$ and $C_q$ are both even cycles.

\item If $n$ is even $C_p$ is an odd cycle and $C_q$ is an even cycle.
\end{itemize}

Throughout the proofs of next lemmas, we define 
$$f(u,v) = \sqrt{\frac{n_{u}+n_{v}-2} {n_{u} n_{v}}},$$ 
for any edge $uv \in E(G),$ and we write $G[H]$ for the subgraph induced in $G$ by the vertex set of graph $H.$ In Lemmas \ref{lem:b1_1}, \ref{lem:01} and \ref{lem:02} we present the $ABC_{GG}$ for all graphs in $B_1(n).$ Note that for a fixed $n$ some non-isomorphic graphs can be obtained by varying $p$ and $q$ such that $n=p+q-1.$ 

\begin{lemma}\label{lem:b1_1}
Let $G$ $\in$ $B_1(n)$ be a graph on $n=p+q-1$ vertices such that $C_p$ and $C_q$ are odd cycles. Then 
\begin{eqnarray}
ABC_{GG}(G) &=& 2 \, {\left(p - 1\right)} \sqrt{\frac{p + q - 4}{{\left(p - 1\right)} {\left(p +2q -3\right)}}} + \frac{2 \, \sqrt{p - 3}}{p - 1} + \nonumber \\
&+& 2 \, {\left(q - 1\right)} \sqrt{\frac{p + q - 4}{{\left(2 \, p + q - 3\right)} {\left(q - 1\right)}}} 
 + \frac{2 \, \sqrt{q - 3}}{q - 1}.
\label{eq:lemma2}
\end{eqnarray}
\end{lemma}

\begin{proof} Let $G$ be the graph labeled as Figure \ref{fig2}. Let $H_{1}=G[C_p]$ be the graph induced  by vertices $\{v_1,v_2,...,v_p\}$ and consider the edge ($v_1,v_p$) $\in$ $E(H_1)$. Note that $n_{v_1}=\frac{p+2q-3}{2}$ and $n_{v_p}=\frac{p-1}{2}$. Taking the advantage of the symmetry, we can observe that this same situation occurs $p-1$ times which can be written as $n_{v_{i}}=\frac{p+2q-3}{2}$ and $n_{v_{i+1}}=\frac{p-1}{2}$ for $i=\left\{1,\ldots,\frac{p-1}{2}\right\}$. For $i$ $\in$   $\left\{\frac{p+3}{2},\ldots,p-1 \right\}$ $n_{v_{i}}=\frac{p-1}{2}$ and $n_{v_{i+1}}=\frac{p+2q-3}{2}$. The remaining edge $(v_{\frac{p+1}{2}},v_{\frac{p+3}{2}})$ has $n_{v_{\frac{p+1}{2}}}=n_{v_{\frac{p+3}{2}}}=\frac{p-1}{2}$. Thus 
\begin{eqnarray}
\sum_{uv \in E(H_1)}f(u,v)&=&(p-1)\sqrt{\frac{\left(\frac{p-1}{2}\right) + \left(\frac{p+2q-3}{2}\right) -2}{ {\left(\frac{p-1}{2}\right)\left(\frac{p+2q-3}{2} \right)  }} } + \sqrt{\frac{\left(\frac{p-1}{2}\right) + \left(\frac{p-1}{2}\right) -2}{ {\left(\frac{p-1}{2}\right)\left(\frac{p-1}{2} \right)  }} } \\
&=& 2 \, {\left(p - 1\right)} \sqrt{\frac{p + q - 4}{{\left(p - 1\right)} {\left(p +2q -3\right)}}} + \frac{2 \, \sqrt{p - 3}}{p - 1}.
\end{eqnarray}

Now, let $H_2=G[C_q]$ be the graph induced  by vertices $\{v_1, v_{p+1}, \ldots, v_n \}.$ Considering the edge ($v_1,v_n$), we get $n_{v_1} = \frac{2p+q-3}{2}$ and $n_{v_{n}} = \frac{q-1}{2}.$ Analogously to the previous case, $n_{v_{j}}=\frac{2p+q-3}{2}$ and $n_{v_{j+1}}=\frac{q-1}{2}$ for $j=\left\{p+1,\ldots,\frac{2p+q-3}{2}\right\}$. For $j$ $\in$ $\left\{\frac{2p+q+1}{2},\ldots,n \right\}$ then $n_{v_{j}}=\frac{q-1}{2}$ and $n_{v_{j+1}}=\frac{2p+q-3}{2}$. The remaining edge $(u,v) =\left(v_{\frac{2p+q-1}{2}},v_{\frac{2p+q+1}{2}} \right)$ has $n_u=n_v=\frac{q-1}{2}$. Thus,
\begin{eqnarray*}
\sum_{uv \in E(H_2)}f(u,v)&=&(q-1)\sqrt{\frac{\left(\frac{q-1}{2}\right) + \left(\frac{2p+q-3}{2}\right) -2}{ {\left(\frac{q-1}{2}\right)\left(\frac{2p+q-3}{2} \right)  }} } +  \sqrt{\frac{\left(\frac{q-1}{2}\right) + \left(\frac{q-1}{2}\right) -2}{ {\left(\frac{q-1}{2}\right)\left(\frac{q-1}{2} \right)  }} }\\
&=& 2 \, {\left(q - 1\right)} \sqrt{\frac{p + q - 4}{{\left(2 \, p + q - 3\right)} {\left(q - 1\right)}}} + \frac{2 \, \sqrt{q - 3}}{q - 1}.
\end{eqnarray*}

Therefore, 

\begin{eqnarray*}
ABC_{GG}(G) &=& \sum_{uv \in H_1}f(u,v) + \sum_{uv \in H_2}f(u,v) \\
&=& 2 \, {\left(p - 1\right)} \sqrt{\frac{p + q - 4}{{\left(p + 2 \, q - 3\right)} {\left(p - 1\right)}}} + \frac{2 \, \sqrt{p - 3}}{p - 1} \\
&+& 2 \, {\left(q - 1\right)} \sqrt{\frac{p + q - 4}{{\left(2 \, p + q - 3\right)} {\left(q - 1\right)}}} 
 + \frac{2 \, \sqrt{q - 3}}{q - 1},
\end{eqnarray*}
and the result follows.

\begin{figure}[!htb]
    \centering
    \includegraphics[height=3.8cm]{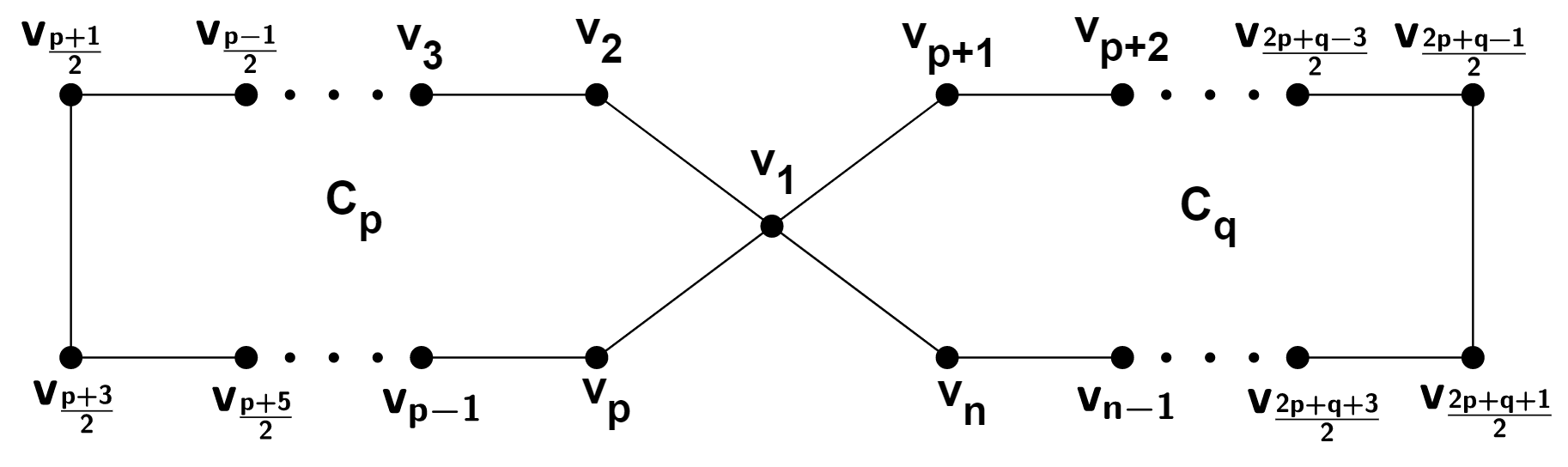}
    \caption{Vertex labeling of a graph $G \, \in \, B_1(n)$ where $C_p$ and $C_q$ are odd cycles}
    \label{fig2}
\end{figure}

\end{proof}

\begin{lemma}
\label{lem:01}
Let $G$ $\in$ $B_1(n)$ be a graph on $\, n = p+q-1\,$ vertices such that $C_p$ and $C_q$ are even cycles. Then 
\begin{eqnarray}\label{eq:lem:01}
ABC_{GG}(G)=  2 \,p \sqrt{\frac{p + q - 3}{p\,{\left(p + 2 \, q - 2\right)}} } + 2 \,q \sqrt{\frac{p + q - 3}{q \, {\left(2 \, p + q - 2\right)} }}
\end{eqnarray}

\end{lemma}
\begin{proof} Let $G$ be the graph labeled as Figure \ref{fig3}. Let $H_{1} = G[C_p]$ be the graph induced  by vertices $\left\{v_1,v_2,\ldots,v_p \right\}$ and consider the edge ($v_1,v_p$) $\in$ $E(H_1)$. Note that $n_{v_1}=\frac{p+2q-2}{2}$ and $n_{v_p}=\frac{p}{2}$. Taking the advantage of the symmetry, we can observe that this same situation occurs $p$ times which can be written as $n_{v_{i}}=\frac{p+2q-2}{2}$ and $n_{v_{i+1}}=\frac{p}{2}$ for $i$ $\in$ $\left\{1,\ldots,\frac{p}{2} \right\}$. For $i$ $\in$ $\left\{\frac{p}{2},\ldots,p-1 \right\}$, $n_{v_{i}}=\frac{p}{2}$ and $n_{v_{i+1}}=\frac{p+2q-2}{2}$. Thus,
\begin{eqnarray*}
\sum_{uv \in H_1}f(u,v) &=& p\sqrt{\frac{\left(\frac{p}{2}\right) + \left(\frac{p+2q-2}{2}\right) -2}{ {\left(\frac{p}{2}\right)\left(\frac{p+2q-2}{2} \right) }} }\\ 
&=&  p\sqrt{ \frac{\left( \frac{2p+2q-2-4}{2}\right)4}{p \, (p+2q-2)}} = 
2 \,p \sqrt{\frac{p + q - 3}{p \, {\left(p + 2 \, q - 2\right)}} }.
\end{eqnarray*}
Now, let $H_2=G[C_q]$ be the induced graph by vertices $\left\{v_1, v_{p+1}, \ldots, v_n \right\}.$ Considering  edge ($v_1,v_n$), we get $n_{v_1} = \frac{2p+q-2}{2}$ and $n_{v_{n}} = \frac{q}{2}.$ Analgously to the previous case, we have $n_{v_{j}}=\frac{2p+q-2}{2}$ and $n_{v_{j+1}}=\frac{q}{2}$ for $j$ $\in$ $\left \{p+1,\ldots,\frac{2p+q-2}{2} \right\}$. For $j$ $\in$ $\left\{\frac{2p+q}{2},\ldots,n-1 \right\}$, we obtain  $n_{v_{j}}=\frac{q}{2}$ and $n_{v_{j+1}}=\frac{2p+q-2}{2}.$ Thus,
\begin{equation*}
\sum_{uv \in E(H_2)}f(u,v)=q\sqrt{\frac{\left(\frac{q}{2}\right) + \left(\frac{q+2p-2}{2}\right) -2}{ {\left(\frac{q}{2}\right)\left(\frac{q+2p-2}{2} \right)  }} } = q\sqrt{ \frac{\left( \frac{2p+2q-2-4}{2}\right)4}{q \, (p+2q-2)}} = 2 \,q \sqrt{\frac{p + q - 3}{{q \, }{\left(2 \, p + q - 2\right)} }}.
\end{equation*}

\begin{figure}[!htb]
    \centering
    \includegraphics[height=4.0cm]{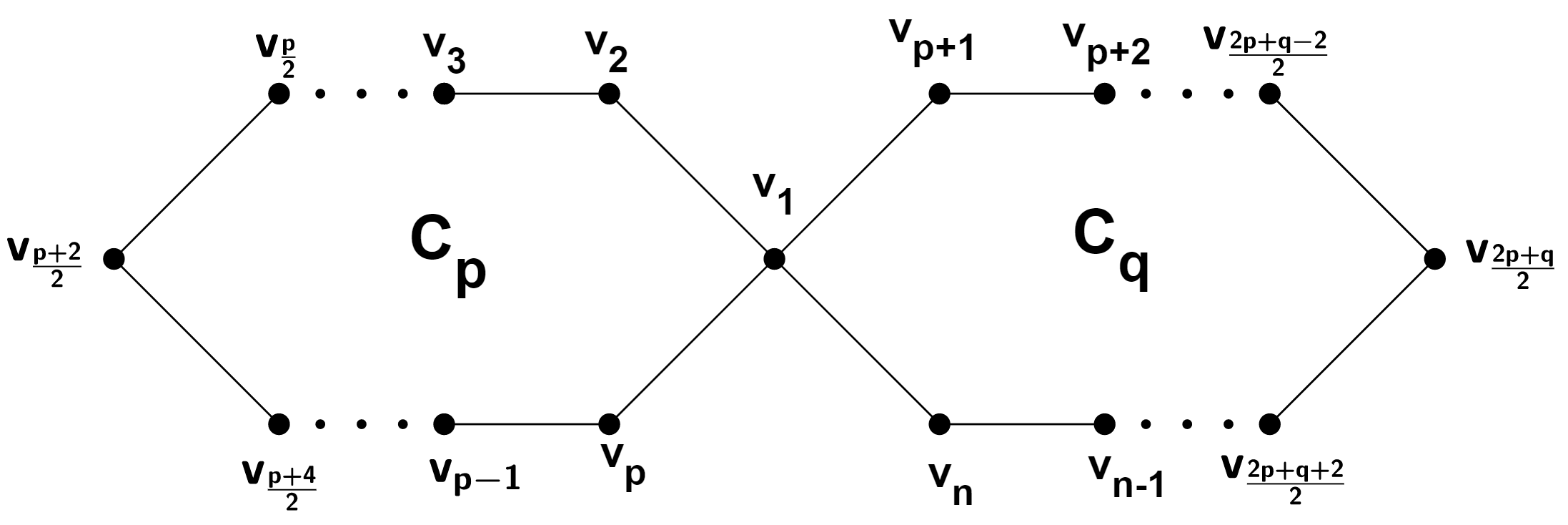}
    \caption{Vertex labeling of a graph $G \, \in \, B_1(n)$ where $C_p$ and $C_q$ are even cycles}
    \label{fig3}
\end{figure}
Therefore, 
\begin{eqnarray*}
ABC_{GG}(G)= \sum_{uv \in H_1}f(u,v) + \sum_{uv \in H_2}f(u,v) = 
2 \,p \sqrt{\frac{p + q - 3}{p\,{\left(p + 2 \, q - 2\right)}} } + 2 \,q \sqrt{\frac{p + q - 3}{{q\, }{\left(2 \, p + q - 2\right)} }}, 
\end{eqnarray*}
and the result follows. 
\end{proof}

\begin{lemma}
\label{lem:02}
Let $G$ $\in$ $B_1(n)$ be a graph on $\,n = p+q-1\,$ vertices such that $C_p$ is an odd cycle and $C_q$ is an even cycle. Then 
\begin{eqnarray}\label{eq:lem:02}
ABC_{GG}(G)= \frac{2 \, \sqrt{p - 3}}{p - 1} + 2 \, {\left(p - 1\right)} \sqrt{\frac{p + q - 4}{{\left(p + 2 \, q - 3\right)} {\left(p - 1\right)}}} + 2 \, q \sqrt{\frac{p + q - 3}{q\, {\left(2 \, p + q - 2\right)} }}.
\end{eqnarray}
\end{lemma}

\begin{proof} 

Let $G$ be the graph labeled as Figure \ref{fig4}. Let $H_{1}=G[C_p]$ be the graph induced by vertices $\{v_1,v_2,\ldots,v_p\}$ and consider the edge ($v_1,v_p$) $\in$ $E(H_1)$. Note that $n_{v_1}=\frac{p+2q-3}{2}$ and $n_{v_p}=\frac{p-1}{2}.$ Taking the advantage of the symmetry, we can observe that this same situation occurs $p-1$ times wich can be written as $n_{v_{i}}=\frac{p+2q-3}{2}$ and $n_{v_{i+1}}=\frac{p-1}{2}$ for $i$ $\in$ $\{1,\ldots,\frac{p-1}{2}\}$. For $i$ $\in$   $\left\{\frac{p+3}{2},\ldots,p-1 \right\}$, then $n_{v_{i}}=\frac{p-1}{2}$ and $n_{v_{i+1}}=\frac{p+2q-3}{2}$. The remaining edge $\left(v_{\frac{p+1}{2}},v_{\frac{p+3}{2}}\right)$ has $n_{v_{\frac{p+1}{2}}}=n_{v_{\frac{p+3}{2}}}=\frac{p-1}{2}$. Thus,
\begin{eqnarray*}
\sum_{uv \in H_1}f(u,v) & = & (p-1)\sqrt{\frac{\left(\frac{p-1}{2}\right) + \left(\frac{p+2q-3}{2}\right) -2}{ {\left(\frac{p-1}{2}\right)\left(\frac{p+2q-3}{2} \right)  }} } +  
\sqrt{\frac{\left(\frac{p-1}{2}\right) + \left(\frac{p-1}{2}\right) -2}{ {\left(\frac{p-1}{2}\right)\left(\frac{p-1}{2} \right)  }} }\\
&=&  2 \, {\left(p - 1\right)} \sqrt{\frac{p + q - 4}{{\left(p - 1\right)} {\left(p +2q -3\right)}}} + \frac{2 \, \sqrt{p - 3}}{p - 1}.
\end{eqnarray*}

Now, let $H_2=G[C_q]$ be the graph induced  by vertices $\{v_1, v_{p+1}, \ldots, v_n \}.$ Considering the edge ($v_1,v_n$), we get $n_{v_1} = \frac{2p+q-2}{2}$ and $n_{v_{n}} = \frac{q}{2}.$ Analgously to the previous case, $n_{v_{j}}=\frac{2p+q-2}{2}$ and $n_{v_{j+1}}=\frac{q}{2}$ for $j$ $\in$ $\left\{p+1,\ldots,\frac{2p+q-2}{2}\right\}$. For $j$ $\in$ $\left\{\frac{2p+q}{2},\ldots,n-1 \right\}$ then $n_{v_{j}}=\frac{q}{2}$ and $n_{v_{j+1}}=\frac{2p+q-2}{2}$. Thus,

\begin{eqnarray*}
\sum_{uv \in E(H_2)} f(u,v) &=& q\sqrt{\frac{\left(\frac{q}{2}\right) + \left(\frac{2p+q-2}{2}\right) -2}{ {\left(\frac{q}{2}\right)\left(\frac{2p+q-2}{2} \right)  }} }\\
&=& 2 \, q \sqrt{\frac{p + q - 3}{q \, {\left(2 \, p + q - 2\right)} }}.
\end{eqnarray*}

Therefore, 

\begin{eqnarray*}
\begin{array}{c}
ABC_{GG}(G)= \sum_{uv \in H_1}f(u,v) + \sum_{uv \in H_2}f(u,v) = 

2 \, {\left(p - 1\right)} \sqrt{\frac{p + q - 4}{{\left(p + 2 \, q - 3\right)} {\left(p - 1\right)}}} + \frac{2 \, \sqrt{p - 3}}{p - 1} + \\ \\ + 2 \, q \sqrt{\frac{p + q - 3}{ q \, \left(2 \, p + q - 2\right)}},
 
\end{array}
\end{eqnarray*}

\noindent and the result follows. 

\begin{figure}[h!]
    \centering
    \includegraphics[height=4cm]{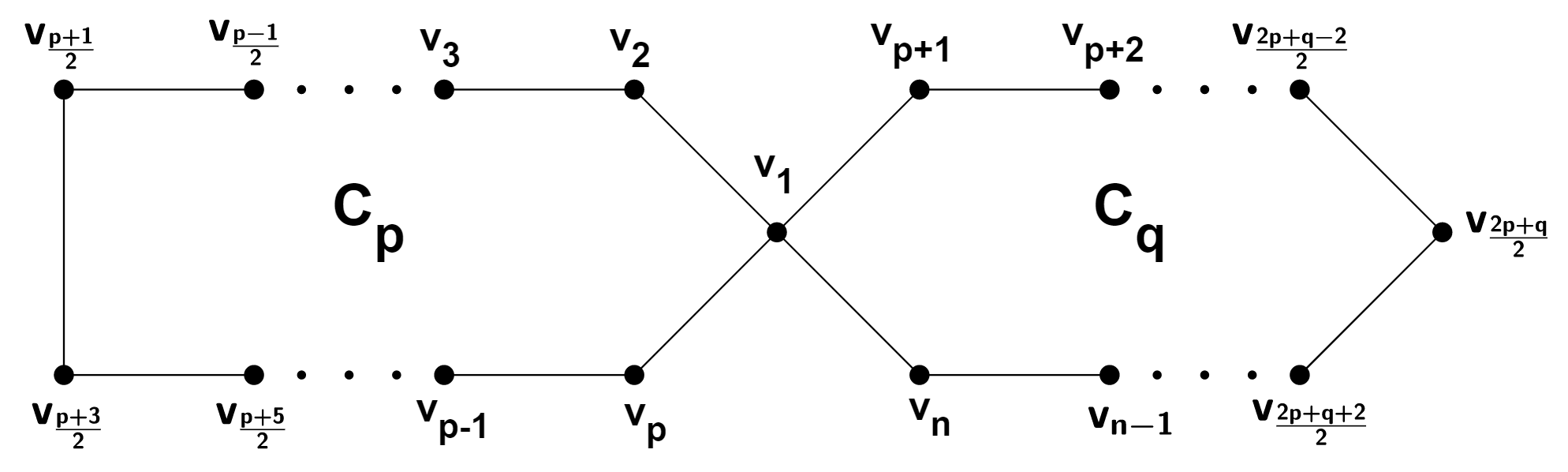}
    \caption{Vertex labeling of a graph $G \, \in \, B_1(n)$ where $C_p$ is an odd cycle and $C_q$ is an even cycle}
    \label{fig4}
\end{figure}

\end{proof}

\section{Minimizing $ABC_{GG}(G)$ for all $G \in B_1(n)$}

In order to characterize the graphs with  minimal Graovac-Ghorbani index in $B_{1}(n),$ we performed  computational experiments to bicyclic graphs with no pendant vertices up to 16  vertices. These graphs were generated by using Nauty-Traces package \cite{McKay2014} and the $ABC_{GG}$ indexes were computed in BlueJ software  \cite{Dragan2019, Barnes2016}. 

\subsection{When $n$ is odd}

\vspace{0.4cm}

We first consider the case when both cycles has odd lenght. Since $n \geq 9$ and $n$ is odd, we have that $n=2k-1$, where  $k \geq 5$ and $k \in \mathbb{N}.$ 

\vspace{0.4cm}

\noindent \textbf{Case $(i):$ $C_p$ and $C_q$ are odd cycles.}

\vspace{0.4cm}

From the symmetry of the process of removing vertices from cycle $C_p$ and adding them to $C_q$, we get that: if $k$ is odd, $3 \leq p \leq \frac{n+1}{2}$ and $\frac{n+1}{2} \leq q \leq n-2;$ if $k$ is even, $3 \leq p \leq \frac{n-1}{2}$ and $\frac{n+3}{2} \leq q \leq n-2.$ From these inequalities, we obtain that: if $k$ is odd, then  $3 \leq p \leq k$ and $k \leq q \leq 2k-3$; if $k$ is even, then  $3 \leq p \leq k-1$ and $k+1 \leq q \leq 2k-3.$

Let $x \in \mathbb{N}$ be the number of vertices removed from $C_q$ and added to $C_p.$ Note that in this process, $x$ should be even to keep both cycles of odd length and that $p=3+x$ and $q=2k-3-x.$ The following facts are true:
\begin{enumerate}
\item[(F1) ] If $k \geq 5$ is odd, then $0 \leq  x \leq k-3;$ 
\item[(F2) ] If $k \geq 5$ is even, then $0 \leq  x \leq k-4.$ 
\end{enumerate}

Let $G^1(p,q)$ be a graph of order $n$  belonging to $B_{1}(p,q)$ with fixed $p$ and $q$ such that $n=p+q-1.$ Our proof initially considers the graph with $p=3$ and $q=n-2,$ which we denote by  $G^1_{0}(3,n-2).$ Removing two vertices from cycle $C_{n-2}$ of the graph $G^1_{0}(3,n-2)$ and adding them to $C_3$, we obtain the graph 
$G^{1}_{1}(5,n-4).$ If $n$ and $k$ are odd, we prove that $G^1_{0}(3,n-2)$ has minimal $ABC_{GG}$ among all graphs $G^1_{x}(3+x,n-2-x),$ where $0 \leq x \leq (n-5)/2.$ If $n$ is odd and $k$ is even, we prove that $G^1_{0}(3,n-2)$ has minimal $ABC_{GG}$ among all graphs $G^1_{x}(3+x,n-2-x),$ where $0 \leq x \leq (n-7)/2.$ 

For instance, for $n=17$ we have $k=9$ and all graphs in $B_1(17)$ are: $G^1_0(3,15),$ $G^1_2(5,13),$ $G^1_4(7,11),$ $G^1_6(9,9).$ In the other hand, if $n=19$, we have $k=10$ and all graphs in $B_1(19)$ are: $G^1_0(3,17),$ $G^1_2(5,15),$ $G^1_4(7,13),$ $G^1_6(9,11).$ In both cases, we will prove in the next lemmas that graph  $G^1_{0}(3,n-2)$ has minimal $ABC_{GG}.$

Recall Equation (\ref{eq:lemma2}) of Lemma \ref{lem:b1_1} and rewrite it as:
\begin{eqnarray}
f_{1}(x):&=& ABC_{GG}(G)= 2 {\left(x + 2\right)} \sqrt{\frac{2k - 4}{{\left(4k - x - 6\right)}{\left(x +2\right)}}}+ \nonumber \\
&+& 2 \, {\left(2 \, k - x - 4\right)} \sqrt{\frac{2k - 4}{{\left(2k +x \right)}{\left(2 \, k - x - 4\right)}}}  + \frac{2 \, \sqrt{2 \, k - x - 6}}{2 \, k - x - 4} + \frac{2 \, \sqrt{x}}{x + 2}. \label{ff1}
\end{eqnarray}

Next, we prove that the first three terms of $f_1(x)$ is an increasing function in $x.$
%
%
\begin{lemma}
\label{l_1}
Let $n=2k-1$ such that $k \geq 5$ and let $x \geq 0.$ Then, the function  
\begin{eqnarray*}
g_{1}(x) &=& 2 {\left(x + 2\right)} \sqrt{\frac{2k - 4}{{\left(4k - x - 6\right)}{\left(x +2\right)}}} + 2 \, {\left(2 \, k - x - 4\right)} \sqrt{\frac{2k - 4}{{\left(2k +x \right)}{\left(2 \, k - x - 4\right)}}} \nonumber \\
&+& \frac{2 \, \sqrt{2 \, k - x - 6}}{2 \, k - x - 4}
\end{eqnarray*}
is increasing in $x.$
\end{lemma}

\begin{proof} We will use Facts (F1) and (F2) to prove our result. We split $g_1(x)$ into two cases: 

\vspace{0.4cm}

\noindent {\bf Case 1.} Let $h(x)= \frac{2 \, \sqrt{2 \, k - x - 6}}{2 \, k - x - 4}.$ We have that $h^{\prime}(x)$ is given by 
$$h^{\prime}(x) = \frac{2 \, k - x - 8}{{\left(2 \, k - x - 4\right)}^{2} \sqrt{2 \, k - x - 6}}.
$$ 

\vspace{0.4cm}

\noindent Assume that $k \geq 5$ is odd. From (F1), we have $2\, k - x - 8 \geq 0$ and $2 \, k - x - 6 \geq 0$.  Now, assume that $k \geq 5$ is even. From (F2), we have $2\, k - x - 8 \geq 0$ and $2 \, k - x - 6 \geq 0$. Then $h^{\prime}(x) > 0$, which means that $h(x)$ is increasing in $x$ for all $k\geq 5$. 

\vspace{0.4cm}

\noindent {\bf Case 2.} Let $m(x) = 2 {\left(x + 2\right)} \sqrt{\frac{2k - 4}{{\left(4k - x - 6\right)}{\left(x +2\right)}}} + 2 \, {\left(2 \, k - x - 4\right)} \sqrt{\frac{2k - 4}{{\left(2k +x \right)}{\left(2 \, k - x - 4\right)}}}$. We have that $m^{\prime}(x)$ is given by \begin{eqnarray*}
m^{\prime}(x) &=& \frac{ 8{\left(k - 1\right)} {\left(k - 2\right)}}{ {\left(x + 2\right)} {\left(4 \, k - x - 6\right)}^{2} \sqrt{\frac{2k - 4}{{\left(x+2\right)} {\left(4 \, k - x - 6\right)}}}} 
- \frac{ 8{\left(k - 1\right)} {\left(k - 2\right)}}{ {\left(2k-x -4\right)} {\left(2k +x\right)}^{2} \sqrt{\frac{2k - 4}{{\left(2k-x-4\right)} {\left(2k +x\right)}}}} \\ 
     &=& \frac{ 8{\left(k - 1\right)} {\left(k - 2\right)}}{\sqrt{(2k-4)}\left( { \frac{ {\left(x + 2\right)} {\left(4 \, k - x - 6\right)}^{2}}{ \sqrt{ {{\left(x+2\right)} {\left(4 \, k - x - 6\right)}}}}  }\right)} 
- \frac{ 8{\left(k - 1\right)} {\left(k - 2\right)}}{\sqrt{(2k-4)}\left({\frac{\left(2k-x -4\right){\left(2k+x\right)}^{2}}{\sqrt{{{\left(2k-x-4\right)} {\left(2k +x\right)}}}} }\right)} 
\end{eqnarray*}

\vspace{0.4cm}

\noindent Define $d_1(x)$ and $d_2(x)$ as

$$d_1(x) = \frac{ {\left(x + 2\right)} {\left(4 \, k - x - 6\right)}^{2}}{ \sqrt{ {{\left(x+2\right)} {\left(4 \, k - x - 6\right)}}}} = (4k-x-6)\sqrt{(4k-x-6)(x+2)},$$ 

$$d_2(x) = \frac{\left(2k-x -4\right){\left(2k +x\right)}^{2}}{\sqrt{{{\left(2k-x-4\right)} {\left(2k +x\right)}}}} = (2k+x)\sqrt{(2k+x)(2k-x-4)}.$$

\vspace{0.4cm}

\noindent If $d_2(x) \geq d_1(x),$ then $m^{\prime}(x) \geq 0$ and $m(x)$ is increasing in $x.$ Let $t(x)=d_2(x)^2 - d_1(x)^2$. By algebraic manipulations, we have that 
\begin{eqnarray*}
t(x) &=& d_2(x)^2 - d_1(x)^2 \\
       &=& (2k +x)^3(2k-x-4) - (x+2)(4k - x - 6)^3 \\ 
       &=& 16 \, k^{4} + 16 \, k^{3} x - 4 \, k x^{3} - x^{4} - 32 \, k^{3} - 48 \, k^{2} x - 24 \, k x^{2} - \\
       && - 4 \, x^{3} \\
       && -(64 \, k^{3} x - 48 \, k^{2}x^{2} + 12 \, k x^{3} - x^{4} + 128 \, k^{3} - 384 \, k^{2} x +\\
       && +168 \, k x^{2} - 20 \, x^{3} - 576 \, k^{2} + 720 \, k x - 144 \, x^{2} + 864 \, k -\\
       && - 432 \, x - 432) \\   
       &=& 16 \, k^{4} - 48 \, k^{3} x + 48 \, k^{2} x^{2} - 16 \, k x^{3} - 160 \, k^{3} + 336 \, k^{2} x - \\ 
       && - 192 \, k x^{2} + 16 \, x^{3} + 576 \, k^{2} - 720 \, kx  + \\
     && + 144 \, x^{2} - 864 \, k + 432 \, x + 432  \\  
t(x) &=& 16(k-1)(k-x-3)^3 .
\end{eqnarray*}

\vspace{0.7cm}

\noindent Note that  $x=k-3$ is the root of $t(x).$ Using Facts (F1) and F(4), we get $t(x) > 0$.  We have that $t(x) = d_2(x)^2 - d_1(x)^2 =(d_2(x)-d_1(x))(d_2(x)+d_1(x)) \geq 0$, then $d_2(x) - d_1(x) \geq 0$. Thus, $m^{\prime}(x) > 0$, which means that $m(x)$ is increasing in $x$ for all $k \geq 5$.


\noindent Therefore, from Cases $1$ and $2$, we have that $g_{1}^{\prime}(x) > 0$, and so $g_1(x)$ is increasing in $x$ for all $k \geq 5$.    
\end{proof}

\begin{lemma}
\label{l_2}
Let $n=2k-1$ such that $k \geq 5.$ Let  $x \geq 0$ and  
\begin{eqnarray*}
f_1(x) &=& 2 {\left(x + 2\right)} \sqrt{\frac{2k - 4}{{\left(4k - x - 6\right)}{\left(x +2\right)}}} + 2 \, {\left(2 \, k - x - 4\right)} \sqrt{\frac{2k - 4}{{\left(2k +x \right)}{\left(2 \, k - x - 4\right)}}} \nonumber \\
 &+& \frac{2 \, \sqrt{2 \, k - x - 6}}{2 \, k - x - 4} + \frac{2 \, \sqrt{x}}{x + 2}. 
\end{eqnarray*}
Then, we have that $f_1(0)$ is the minimal value of  $f_1(x)$.
\end{lemma}
\begin{proof}

Let $g_1(x)$ be defined by 
\begin{eqnarray*}
g_1(x) &=& 2 {\left(x + 2\right)} \sqrt{\frac{2k - 4}{{\left(4k - x - 6\right)}{\left(x +2\right)}}} + 2 \, {\left(2 \, k - x - 4\right)} \sqrt{\frac{2k - 4}{{\left(2k +x \right)}{\left(2 \, k - x - 4\right)}}} \nonumber \\
 &+&  \frac{2 \, \sqrt{2 \, k - x - 6}}{2 \, k - x - 4}.
\end{eqnarray*}
From Lemma \ref{l_1} we have that $g_1(x)$ is increasing in $x$ for all $k\geq 5$, which implies that 
$$g_1(x) \geq g_1(0).$$
So, 
$f_1(x) \geq g_1(x) \geq g_1(0) = f_1(0),$
and the result follows.
\end{proof}

Note that $f_1(0) = ABC_{GG}(G_{0}^{1}(3,n-2))$ and Lemma \ref{l_2} implies that $f_1(x)$ is minimized to the graph $G_{0}^{1}(3,n-2)$ among all graphs $G_{i}^{1}(3,n-2)$ when $C_p$ and $C_q$ have odd lengths.

\vspace{0.4cm}

\noindent \textbf{Case $(ii):$ $C_p$ and $C_q$ are even cycles.}

\vspace{0.4cm}

From the symmetry of the process of removing vertices from cycle $C_p$ and adding them to $C_q$, we get that: if $k$ is odd, $4 \leq p \leq \frac{n-1}{2}$ and $\frac{n+3}{2} \leq q \leq n-3;$ if $k$ is even, $4 \leq p \leq \frac{n+1}{2}$ and $\frac{n+1}{2} \leq q \leq n-3.$ From these inequalities, we get: if $k$ is odd,  $4 \leq p \leq k-1$ and $k+1 \leq q \leq 2k-4$; if $k$ is even,  $4 \leq p \leq k$ and $k \leq q \leq 2k-4$. 

Let $x \in \mathbb{N}$ be the number of vertices removed from $C_q$ and added to $C_p.$ Note that in this process, $x$ should be even to keep both cycles of even length and that $p=4+x$ and $q=2k-4-x$. The following facts are true:
\begin{enumerate}
\item[(F3) ] If $k \geq 5$ is odd, then $0 \leq  x \leq k-5;$ 
\item[(F4) ] If $k \geq 5$ is even, then $0 \leq  x \leq k-4.$ 
\end{enumerate}

We will prove that for $k \geq 11,$ the graph $G^{1}_{0}(4,n-3)$ has minimal $ABC_{GG}$ index among all graph when $C_{p}$ and $C_{q}$ are both even cycles.

So, we can rewrite Equation (\ref{eq:lem:01}) of Lemma \ref{lem:01} as:
\begin{eqnarray}\label{ff1}
f_{2}(x):= 2 \, {\left(x + 4\right)} \sqrt{\frac{2 \, k - 3}{{\left(4 \, k - x - 6\right)} {\left(x + 4\right)}}}+ 2 \, {\left(2 \, k - x - 4\right)} \sqrt{\frac{2 \, k - 3}{{\left(2 \, k + x + 2\right)} {\left(2 \, k - x - 4\right)}}}. 
\end{eqnarray}

\vspace{1cm}

\begin{lemma} \label{l_3}
Let $n=2k-1$ with $k \geq 5.$ Let $x \geq 0$ and 
\begin{eqnarray}
f_{2}(x)= 2 \, {\left(x + 4\right)} \sqrt{\frac{2 \, k - 3}{{\left(4 \, k - x - 6\right)} {\left(x + 4\right)}}}+ 2 \, {\left(2 \, k - x - 4\right)} \sqrt{\frac{2 \, k - 3}{{\left(2 \, k + x + 2\right)} {\left(2 \, k - x - 4\right)}}}. \nonumber 
\end{eqnarray}
If $k \geq 11,$ then $f_{2}(0)$ is the minimal value of $f_{2}(x)$. If $k$ is odd and $5\leq k \leq 10$, then $f_{2}(k-5)$ is the minimal value of $f_{2}(x).$ If $k$ is even and $5\leq k \leq 10$, then $f_{2}(k-4)$ is the minimal value of $f_{2}(x).$
\end{lemma}

\begin{proof}
The functions $f_{2}^{\prime}(x)$ and $f_{2}^{\prime \prime}(x)$ are given by
\begin{eqnarray*}
f_{2}^{\prime}(x) &=& \frac{2 \, {\left(2 \, k - 1\right)} {\left(2 \, k - 3\right)}}{{\left(4 \, k - x - 6\right)}^{2} {\left(x + 4\right)} \sqrt{\frac{2 \, k - 3}{{\left(4 \, k - x - 6\right)} {\left(x + 4\right)}}}}  -\frac{2 \, {\left(2 \, k - 1\right)} {\left(2 \, k - 3\right)}}{{\left(2 \, k + x + 2\right)}^{2} {\left(2 \, k - x - 4\right)} \sqrt{\frac{2 \, k - 3}{{\left(2 \, k + x + 2\right)} {\left(2 \, k - x - 4\right)}}}} \nonumber \\  
f_{2}^{\prime \prime}(x) &=& \frac{2 \, {\left(2 \, k - 1\right)} {\left(2 \, k - 3\right)}^{2} {\left(2 \, k - 2x - 7\right)}}{{\left(2 \, k + x + 2\right)}^{4} {\left(2 \, k - x - 4\right)}^{3} \left(\frac{2 \, k - 3}{{\left(2 \, k + x + 2\right)} {\left(2 \, k - x - 4\right)}}\right)^{\frac{3}{2}}} - \frac{2{\left(2 \, k - 1\right)} {\left(2 \, k - 3\right)}^{2} {\left(2 \, k - 2x - 9\right)}}{(x+4)^3 \, {\left(4 \, k - x - 6\right)}^{4} \left(\frac{2 \, k - 3}{\left(x+4\right)( 4k - x - 6) }\right)^{\frac{3}{2}}} \nonumber
\end{eqnarray*}

Taking $f_{2}^{\prime}(x) = 0$, we obtain $4(k-x-4)(4kx^2-3x^2-8k^2x+38kx-24x+4k^3-44k^2+100k-52)=0,$ 
which has critical points $x_1=k-4$, $x_2= k-4 + \frac{(3k - 2)}{\sqrt{\left(4k - 3\right)}}$ and $x_3= k-4  - \frac{(3k - 2)}{\sqrt{\left(4k - 3\right)}}$. 
Suppose that $x_3$ is integer. In this case, $(4k-3)$ should be a perfect square, that is, $4k-3 = m^2$. 
Thus, $\frac{(3k-2)}{\sqrt{4k-3}}=\frac{\frac{3m^2}{4} + \frac{1}{4}}{m}=\frac{3m}{4} + \frac{1}{4m}$ cannot be integer for $m>1,$ and we get a contradiction. So, the only critical points are $x = 0$ and $x=k-4.$
We get that 
$$f_{2}^{\prime \prime}(k-4) = \frac{4 \, {\left(2 \, k - 1\right)} {\left(2 \, k - 3\right)}^{2}}{{\left(3 \, k - 2\right)}^{4} k^{3} \left(\frac{2 \, k - 3}{k{\left(3 \, k - 2\right)}}\right)^{\frac{3}{2}}},$$
is positive, and so $x=k-4$ is a minimum of the function $f_2(x).$ 

Let $k \geq 5$ and $k$ is odd. In this case, $0 \leq x \leq k-5$ and we need to prove whether $f_2(0) > f_2(k-5)$ or $f_2(0) < f_2(k-5).$ Let 
\begin{eqnarray*}
m(k) = f_2(k-5) - f_2(0) &=& 2 \, {\left(k + 1\right)} \sqrt{\frac{2 \, k - 3}{3{\left(k + 1\right)} {\left(k - 1\right)}}} + 2 \, {\left(k - 1\right)} \sqrt{\frac{2 \, k - 3}{{\left(3 \, k - 1\right)} {\left(k - 1\right)}}}- \nonumber \\
&&  - 2 \, {\left(k - 2\right)} \sqrt{\frac{2 \, k - 3}{{\left(k + 1\right)} {\left(k - 2\right)}}} - 2\sqrt{2}
\end{eqnarray*}

The derivatives of $m(k)$ are given as follows:
\begin{eqnarray*}
\begin{array}{l}

m^{\prime}(k)=\frac{2 \, \sqrt{3} {\left(k^{2} - 2 \, k + 2\right)}}{3 \, \sqrt{{\left(2 \, k - 3\right)} {\left(k + 1\right)} {\left(k - 1\right)}} {\left(k - 1\right)}} + \frac{2 \, {\left(3 \, k^{2} - 2 \, k - 2\right)}}{\sqrt{{\left(3 \, k - 1\right)} {\left(2 \, k - 3\right)} {\left(k - 1\right)}} {\left(3 \, k - 1\right)}} - \frac{2 \, k^{2} + 4 \, k - 13}{\sqrt{{\left(2 \, k - 3\right)} {\left(k + 1\right)} {\left(k - 2\right)}} {\left(k + 1\right)}} \\ \\ 

m^{\prime \prime}(k) = \frac{4 \, k^{4} + 16 \, k^{3} - 156 \, k^{2} + 316 \, k - 191}{2 \, \sqrt{{\left(2 \, k - 3\right)} {\left(k + 1\right)} {\left(k - 2\right)}} {\left(2 \, k - 3\right)} {\left(k + 1\right)}^{2} {\left(k - 2\right)}} -\frac{2 \, \sqrt{3} {\left(k^{4} - 4 \, k^{3} + 12 \, k^{2} - 10 \, k - 2\right)}}{3 \, \sqrt{{\left(2 \, k - 3\right)} {\left(k + 1\right)} {\left(k - 1\right)}} {\left(2 \, k - 3\right)} {\left(k + 1\right)} {\left(k - 1\right)}^{2}} - \\ \\ -\frac{2 \, {\left(9 \, k^{4} - 12 \, k^{3} - 36 \, k^{2} + 78 \, k - 38\right)}}{\sqrt{{\left(3 \, k - 1\right)} {\left(2 \, k - 3\right)} {\left(k - 1\right)}} {\left(3 \, k - 1\right)}^{2} {\left(2 \, k - 3\right)} {\left(k - 1\right)}}
\end{array}
\end{eqnarray*}

We know that $m(k)$ and $m^{\prime}(k)$ are continuous in this interval. Using a numerical method, we get $k_1=5$ and $k_2 \simeq 10.3147$ as exact real roots for $m(k).$ Using a numerical method for $m^{\prime}(k)$, we obtain  
$k^{\prime}_1 \simeq 7.5248$. Evaluating $m^{\prime\prime}(k)$ for $k^{\prime}_1,$ 
we have that $m^{\prime\prime}(7.5248)>0,$ and $k^{\prime}_1$ is an absolute minimum. Since $m(k)<0$ for $5< k \leq 10$, $h(k^{\prime}_1)<0$ and $h(11)>0$, then, we have that $m(k) \geq 0$ for all $k\geq 10.3147$. Therefore, $m(k) \geq 11$, which implies that $f_2(x) \geq f_2(0)$, and the graph $G_{0}(4,n-3)$ minimizes $f_2(x).$
Now, let $k$ be odd and $5\leq k \leq 10.$ In this case, $m(k)\leq 0,$ which implies that $f_2(k-5)$ is minimum and the graph $G^{1}_{k-5}(p,q) = G^{1}_{k-5}(k-1,k+1)$ minimizes $f_{2}(x).$

If $k$ is even, since $0 \leq x \leq k-4$, we need to prove that $f_{2}(0)<f_{2}(k-4).$
Replacing the extremes $x_1=0$ and $x_2=k-4$ in $f_{2}(x)$ we obtain
$$f_{2}(0) = 2 \, {\left(k - 2\right)} \sqrt{\frac{2 \, k - 3}{{\left(k + 1\right)} {\left(k - 2\right)}}} + 2\sqrt{2},$$

$$f_{2}(k-4) = 4 \, k \sqrt{\frac{2 \, k - 3}{k{\left(3 \, k - 2\right)}}}.$$

\noindent Consider the function 
\begin{eqnarray*}
h(k) &=&  f_{2}(k-4) - f_{2}(0) = 4 \, k \sqrt{\frac{2 \, k - 3}{k{\left(3 \, k - 2\right)}}} - 2 \, {\left(k - 2\right)} \sqrt{\frac{2 \, k - 3}{{\left(k + 1\right)} {\left(k - 2\right)}}} - 2\sqrt{2} \nonumber \\
\end{eqnarray*}
The derivatives of $h(k)$ are given as follows:
\begin{eqnarray*}
\begin{array}{l}
h^{\prime}(k)= \frac{4 \, {\left(3 \, k^{2} - 4 \, k + 3\right)}}{\sqrt{{\left(3 \, k - 2\right)} {\left(2 \, k - 3\right)} k} {\left(3 \, k - 2\right)}} - \frac{2 \, k^{2} + 4 \, k - 13}{\sqrt{{\left(2 \, k - 3\right)} {\left(k + 1\right)} {\left(k - 2\right)}} {\left(k + 1\right)}} \\ \\ 
h^{\prime \prime}(k)=\frac{4 \, k^{4} + 16 \, k^{3} - 156 \, k^{2} + 316 \, k - 191}{2 \, \sqrt{{\left(2 \, k - 3\right)} {\left(k + 1\right)} {\left(k - 2\right)}} {\left(2 \, k - 3\right)} {\left(k + 1\right)}^{2} {\left(k - 2\right)}} - \frac{12 \, {\left(3 \, k^{4} - 8 \, k^{3} + 18 \, k^{2} - 18 \, k + 3\right)}}{\sqrt{{\left(3 \, k - 2\right)} {\left(2 \, k - 3\right)} k} {\left(3 \, k - 2\right)}^{2} {\left(2 \, k - 3\right)} k}.
\end{array}
\end{eqnarray*}
We know that $h(k)$ and $h^{\prime}(k)$ are continuous in this interval. Using a numerical method, we get $k_1 =  \frac{21+5\sqrt{17}}{4} \simeq 10.4039$ and $k_2=4$ (with multiplicity 2) as exact real roots for $h(k).$ 
Making $h^{\prime}(k)=0$, we obtain 
$${\left(36 \, k^{7} - 312 \, k^{6} + 172 \, k^{5} + 1600 \, k^{4} - 2595 \, k^{3} + 1546 \, k^{2} - 332 \, k + 72\right)} {\left(2 \, k - 3\right)} {\left(k - 4\right)} = 0,$$
and the critical points are $k^{\prime}_1 = 4$ and $k^{\prime}_2 = 7.3648$. Evaluating $h^{\prime\prime}(k)$ in $k^{\prime}_2,$ 
we have that $h^{\prime\prime}(7.3648)>0,$ and $k_2$ is an absolute minimum. Since $h(3)<0, h(8)<0, h(k^{\prime}_2)<0$ and $h(11)>0,$ we have that $h(k) \geq 0$ for all $k \geq \frac{21+5\sqrt{17}}{4}.$ Therefore, $h(k)>0$ for all $k\geq 11,$ which implies that $f_{2}(x) \geq f_{2}(0),$ and the graph $G^{1}_{0}(4,n-3)$ minimizes $f_2(x).$
Now, let $k$ be even and $6\leq k \leq 10.$ In this case, $h(k)<0,$ which implies that $f_2(k-4)$ is minimum and the graph $G^{1}_{k-4}(p,q) = G^{1}_{k-4}(k,k)$ minimizes $f_{2}(x).$

\end{proof}

The next result shows that $G_{0}^{1}(3,n-2)$ has minimal $ABC_{GG}$ among all graphs $G$ in $B_{1}(n)$ when $n$ is odd.

\begin{lemma}
\label{l_4} Let $G \in B_{1}(n)$ with odd $n$ and $n \geq 9$. Then, $$ABC_{GG}(G) \geq ABC_{GG}(G_{0}^{1}(3,n-2)).$$ 
\end{lemma}

\begin{proof}

Let $n=2k-1$ such that $n \geq 9.$ First, suppose that $k\geq 11.$ From Lemmas \ref{l_2} and \ref{l_3}, we should prove that $f_{1}(0)<f_{2}(0).$
 Considering $h(k) = f_{2}(0) - f_{1}(0),$ we have that
\begin{eqnarray*}
h(k) &=& 2 \, {\left(k - 2\right)} \sqrt{\frac{2 \, k - 3}{{\left(k + 1\right)} {\left(k - 2\right)}}} + 2\sqrt{2} - \left(  \frac{2 \, \sqrt{2} {\left(k - 2\right)}}{\sqrt{k}} + 2 \sqrt{\frac{2k - 4}{2k - 3}} + \frac{\sqrt{2 \, k - 6}}{k - 2}  \right)  \\ 
h(k) &=& 2 \, {\left(k - 2\right)} \sqrt{\frac{2 \, k - 3}{{\left(k + 1\right)} {\left(k - 2\right)}}} + \frac{\sqrt{2}}{5} + \frac{9\sqrt{2}}{5} - \left(  \frac{2 \, \sqrt{2} {\left(k - 2\right)}}{\sqrt{k}} + 2 \sqrt{\frac{2k - 4}{2k - 3}} + \frac{\sqrt{2 \, k - 6}}{k - 2}  \right)  \\ 
h(k) &=& \underbrace{2 \, {\left(k - 2\right)} \sqrt{\frac{2 \, k - 3}{{\left(k + 1\right)} {\left(k - 2\right)}}} + \frac{\sqrt{2}}{5} - \frac{2 \, \sqrt{2} {\left(k - 2\right)}}{\sqrt{k}}}_{i} + \underbrace{ \frac{9\sqrt{2}}{5} - 2 \sqrt{\frac{2k - 4}{2k - 3}} - \frac{\sqrt{2 \, k - 6}}{k - 2} }_{ii} 
\end{eqnarray*}
\noindent \emph{Case (i).} Let $m(k)$ be defined as 
\begin{eqnarray*}
m(k) &=& 2 \, {\left(k - 2\right)} \sqrt{\frac{2 \, k - 3}{{\left(k + 1\right)} {\left(k - 2\right)}}} + \frac{\sqrt{2}}{5} - \frac{2 \, \sqrt{2} {\left(k - 2\right)}}{\sqrt{k}} \\ 
&=& \frac{\sqrt{k}\sqrt{2k-3}(10k-20) -\sqrt{2}\sqrt{k-2}\sqrt{k+1}(10k-\sqrt{k}-20) }{5\sqrt{k-2}\sqrt{k}\sqrt{k+1}}.
\end{eqnarray*}

\vspace{0.4cm}

\noindent Let $t(k) = d_1(k)^2 - d_2(k)^2$, where $d_1(k)= \sqrt{k}\sqrt{2k-3}(10k-20)$  and \\ \mbox{$d_2(k)=-\sqrt{2}\sqrt{k-2} \sqrt{k+1}(10k-\sqrt{k}-20).$} We need to verify if $t(k)>0$. Then, we have that, 
\begin{eqnarray*}
t(k) &=& d_1(k)^2 - d_2(k)^2\\ 
     &=& k(2k-3)(10k-20)^2-2(k-2)(k+1)(10k-\sqrt{k}-20)^2 \\ 
    &=& 2(k-2)(20k^2\sqrt{k} -51k^2-20k\sqrt{k}+299k-40\sqrt{k}-400)
\end{eqnarray*}
Let $\sqrt{k}=u \geq 0$ and $r(k)=20k^2\sqrt{k} -51k^2-20k\sqrt{k}+299k-40\sqrt{k}-400$. Then, $r(u) =20u^5 -51u^4 -20u^3 +299u^2-40u-400$. We have that $r(u) \geq 0$ for all $u \geq u_1 \simeq 1.42596$ (the unique real root). Then, $r(k) \geq 0$ for all  $k \geq k_1 \simeq 2.03336$. Since $k\geq 5,$  we get $t(k)>0$ and consequently $m(k)>0$ for all $k \geq 5.$ 

\noindent \emph{Case (ii).} Let $g(k)$ be defined as 
\begin{eqnarray*}
g(k) &=&  \frac{9\sqrt{2}}{5} - 2 \sqrt{\frac{2k - 4}{2k - 3}} - \frac{\sqrt{2 \, k - 6}}{k - 2} \\
&=&\frac{\sqrt{2}\sqrt{2k-3}(9k-5\sqrt{k-3}-18) -10\sqrt{k-2}\sqrt{2}(k-2)}{5(k-2)\sqrt{2k-3}}.
\end{eqnarray*}

\noindent Let $t(k) = d_1(k)^2 - d_2(k)^2$, where $d_1(k)= 2(2k-3)(9k-5\sqrt{k-3}-18)^2$ and\\ \mbox{$d_2(k)= -10\sqrt{k-2}\sqrt{2}(k-2).$} We need to check whether $t(k) \geq 0$. Note that 
\begin{eqnarray*}
t(k) &=& d_1(k)^2 - d_2(k)^2 \\ 
 &=&  2(2k-3)(9k-5\sqrt{k-3}-18)^2 -200(k-2)^3 \\ 
 &=& 2(62k^3-180k^2\sqrt{k-3}-241k^2+630k\sqrt{k-3}+195k-540\sqrt{k-3}+53).
\end{eqnarray*}
\noindent By making the variable change $u = \sqrt{k-3}$ we can prove that $t(k) \geq 0$ for all $k \geq 3,$ which implies that $g(k) \geq 0.$ Therefore, from Cases $(i)$ and $(ii)$, we have that $h(k)>0.$ 

Let $k$ be even such that $5 \leq k \leq 10.$ From Lemmas \ref{l_2} and \ref{l_3}, we should prove that $f_{1}(0)<f_{2}(k-4).$ Let $h(k) = f_2(k-4) - f_1(0)$ defined as 
\begin{eqnarray*}
h(k) &=& 4 \, k \sqrt{\frac{2 \, k - 3}{{\left(3 \, k - 2\right)} k}} - \frac{2 \, \sqrt{2} {\left(k - 2\right)}}{\sqrt{k}} - 2 \, \sqrt{\frac{2k - 4}{2 \, k - 3}} - \frac{\sqrt{2 \, k - 6}}{k - 2}.
\end{eqnarray*}
\noindent By replacing each $k$ in $h(k)$ we obtain that $h(k) \geq 0.$ 

Let $k$ be odd such that $5 \leq k \leq 10.$ From Lemmas \ref{l_2} and \ref{l_3}, we should prove that $f_{1}(0)<f_{2}(k-5).$ Let $h(k) = f_2(k-5) - f_1(0)$ defined as

\begin{eqnarray*}
h(k) &=& 2 \, {\left(k + 1\right)} \sqrt{\frac{2 \, k - 3}{3{\left(k + 1\right)} {\left(k - 1\right)}}} + 2 \, {\left(k - 1\right)} \sqrt{\frac{2 \, k - 3}{{\left(3 \, k - 1\right)} {\left(k - 1\right)}}} - \frac{2 \, \sqrt{2} {\left(k - 2\right)}}{\sqrt{k}} - \nonumber \\ \\
&& -2 \, \sqrt{\frac{2k - 4}{2 \, k - 3}} - \frac{\sqrt{2 \, k - 6}}{k - 2}.
\end{eqnarray*}
\noindent By replacing each $k$ in $h(k)$ we obtain that $h(k) \geq 0.$

Now, the proof is complete.

\end{proof}

\subsection{When $n$ is even}

Note that when $n=2k$ such that $n=p+q-1,$ We should have that one of the cycles has odd lenght and other has even lenght. Recall the graph $G^1_{0}(3,n-2).$ 
From the symmetry of the process of removing vertices from cycle $C_q$ and adding them to
$C_q$, we get that $3 \leq p \leq n-3$ and $4 \leq q \leq n-2,$ which imply that
\begin{itemize}
\item [(F5) ] $3 \leq p \leq 2k-3$ and $4 \leq q \leq 2k-2$. 
\end{itemize}

Let $x \in \mathbb{N}$ be the number of vertices removed from $C_q$ and added to $C_p.$ Note that in this process, $x$ should be even to keep both cycles of even length such that $p=3+x,$ $q=2k-2-x\,$ and $$0 \leq  x \leq 2k-6.$$ 
Next, we prove that $G^1_{0}(3,n-2)$ has minimal  $ABC_{GG}(G)$ among all graphs $G \in B_{1}(n)$ with $n$ even. Equation (\ref{eq:lem:02}) of Lemma \ref{lem:02} can be rewrittem as a function of $n$ and $x:$ 
\begin{eqnarray*}
f_{2}(x)= 2 \, {\left(n - x - 2\right)} \sqrt{\frac{n - 2}{{\left(n + x + 2\right)} {\left(n - x - 2\right)}}} + 2 \, {\left(x + 2\right)} \sqrt{\frac{n - 3}{{\left(2 \, n - x - 4\right)} {\left(x + 2\right)}}} + \frac{2 \, \sqrt{x}}{x + 2}.
\end{eqnarray*}
Note that $f(0)$ is equal to $ABC_{GG}(G^1_{0}(3,n-2)).$

\begin{lemma} \label{l_5} 
Let $ n \geq 5$ and $x \geq 0.$ Let $g_2(x)$ be defined as 
\begin{eqnarray*}
g_{2}(x)= 2 \, {\left(n - x - 2\right)} \sqrt{\frac{n - 2}{{\left(n + x + 2\right)} {\left(n - x - 2\right)}}} + 2 \, {\left(x + 2\right)} \sqrt{\frac{n - 3}{{\left(2 \, n - x - 4\right)} {\left(x + 2\right)}}}.
\end{eqnarray*}
Then, $g_2(x)$ has its minimum value in  $g_2(0).$ 

\end{lemma}
\begin{proof}
Let $n \geq 5, x_1 = x$ and $x_2 = x+2.$ In order to prove that $g_2(x)$ is increasing in $x,$ we need to prove that $g_2(x_2)-g_2(x_1) \geq 0.$ Take $h(x) = g_2(x_2) - g_2(x_1)$. Note that
\begin{eqnarray*}
 h(x) &=& -2 \, {\left(n - x - 2\right)} \sqrt{\frac{n - 2}{{\left(n + x + 2\right)} {\left(n - x - 2\right)}}} 
+ 2 \, {\left(n - x - 4\right)} \sqrt{\frac{n - 2}{{\left(n + x + 4\right)} {\left(n - x - 4\right)}}} \nonumber \\ 
&+& 2 \, {\left(x + 4\right)} \sqrt{\frac{n - 3}{{\left(2 \, n - x - 6\right)} {\left(x + 4\right)}}} - 2 \, {\left(x + 2\right)} \sqrt{\frac{n - 3}{{\left(2 \, n - x - 4\right)} {\left(x + 2\right)}}}.
\end{eqnarray*}
One can prove that 
$$h(0) = 2 \, {\left(n - 4\right)} \sqrt{\frac{n - 2}{{\left(n + 4\right)} {\left(n - 4\right)}}} + 2 \, \sqrt{2} - \frac{2 \, {\left(n - 2\right)}}{\sqrt{n + 2}} - 2 \, \sqrt{\frac{n - 3}{n - 2}} \geq 0,$$
and 
$$h(n-6)=-2 \, {\left(n - 4\right)} \sqrt{\frac{n - 3}{{\left(n + 2\right)} {\left(n - 4\right)}}} + 2 \, {\left(n - 2\right)} \sqrt{\frac{n - 3}{{\left(n - 2\right)} n}} - 2 \, \sqrt{2} + 2 \, \sqrt{\frac{n - 2}{n - 1}}\geq 0.$$ By using numerical analysis, we get that polynomial $h(x)$ has no root for $0 \leq x \leq n-6.$ So, $h(x) \geq 0,$ which implies that $g_2(x)$ is increasing and the result follows.
\end{proof}

\begin{lemma}
\label{l_6}

Let $n \geq 5$ and $x\geq 0.$ Let $f_2(x)$ be defined as 
\begin{eqnarray}
f_{2}(x) =  2 \, {\left(n - x - 2\right)} \sqrt{\frac{n - 2}{{\left(n + x + 2\right)} {\left(n - x - 2\right)}}} + 2 \, {\left(x + 2\right)} \sqrt{\frac{n - 3}{{\left(2 \, n - x - 4\right)} {\left(x + 2\right)}}} + \frac{2 \, \sqrt{x}}{x + 2}.
\end{eqnarray}
Then, $f_2(x)$ has its minimum  value in $f_2(0)$ for all $n \geq 5$.
\end{lemma}
\begin{proof}
Note that $f_2(x) \geq g_2(x).$
From Lemma \ref{l_5},
$g_2(x) \geq g_2(0).$ Therefore,  
$f_2(x)  \geq g_2(x) \geq g_2(0) = f_2(0)$, and the result follows.
\end{proof}

Next, we state the main result of this paper.

\begin{theorem}
\label{th1}
Let $G \in B_{1}(n)$ be a graph of order $n \geq 9.$ If $n$ is odd, then
$$ABC_{GG}(G) \geq \frac{2(n-3)}{\sqrt{n-1}} + 2 \sqrt{\frac{n-3}{n-2}} + \frac{2\sqrt{n-5}}{n-3}.$$

If $n$ is even, then 
$$ABC_{GG}(G) \geq 2 \, \sqrt{\frac{n-3}{n-2}} + \frac{2 \, {\left(n - 2\right)}}{\sqrt{n+2}}.$$
Equality holds in both cases  if and only if $G \cong G^{1}_{0}(3,n-2)$.
\end{theorem}
\begin{proof}
Let $G \in B_{1}(n).$ Suppose that $n$ is odd. From Lemma \ref{l_4}, $ABC_{GG}(G) = f(x) \geq f(0) =  ABC_{GG}(G_{0}^{1}(3,n-2)).$  
Now, suppose that $n$ is even. From Lemma \ref{l_6}, $ABC_{GG}(G) = f^{{\prime}\,{\prime}}(x) \geq f^{{\prime}\,{\prime}}(0) =  ABC_{GG}(G_{0}^{1}(3,n-2))$ and the result follows. 

\end{proof}

In Figure \ref{B1_geralf}, the extremal graph $G_{0}^{1}(3,8)$ is displayed. 

\begin{figure}[!h]
    \centering
    \includegraphics[height=3.0cm]{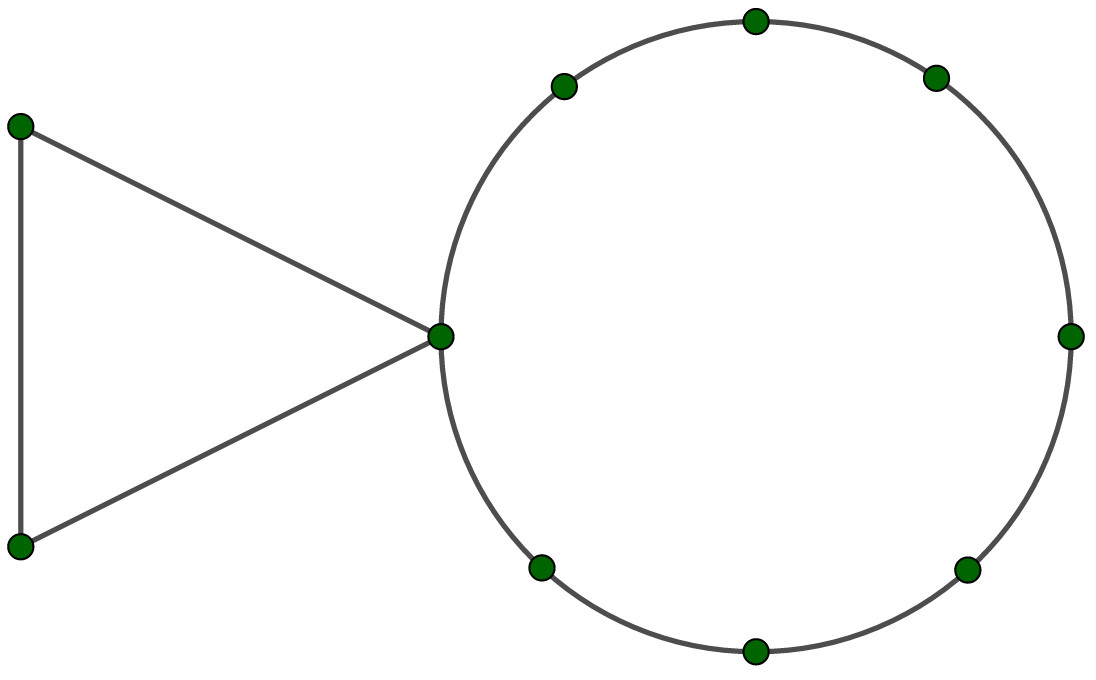}
    \caption{Graph of $B_1(n)$ family with minimal value of the $ABC_{GG}$ index for  $n = 10$.}
    \label{B1_geralf}
\end{figure}

Note that Theorem \ref{th1} states the extremal graphs for $n \geq 9.$ In Figure \ref{B1_n10f}, we display all graphs up to 10 vertices that are extremal to the $ABC_{GG}$ index in the family $B_{1}(n)$ obtained by exhaustive computational search. 

\begin{figure}[!h]
    \centering
    \includegraphics[height=7.0cm]{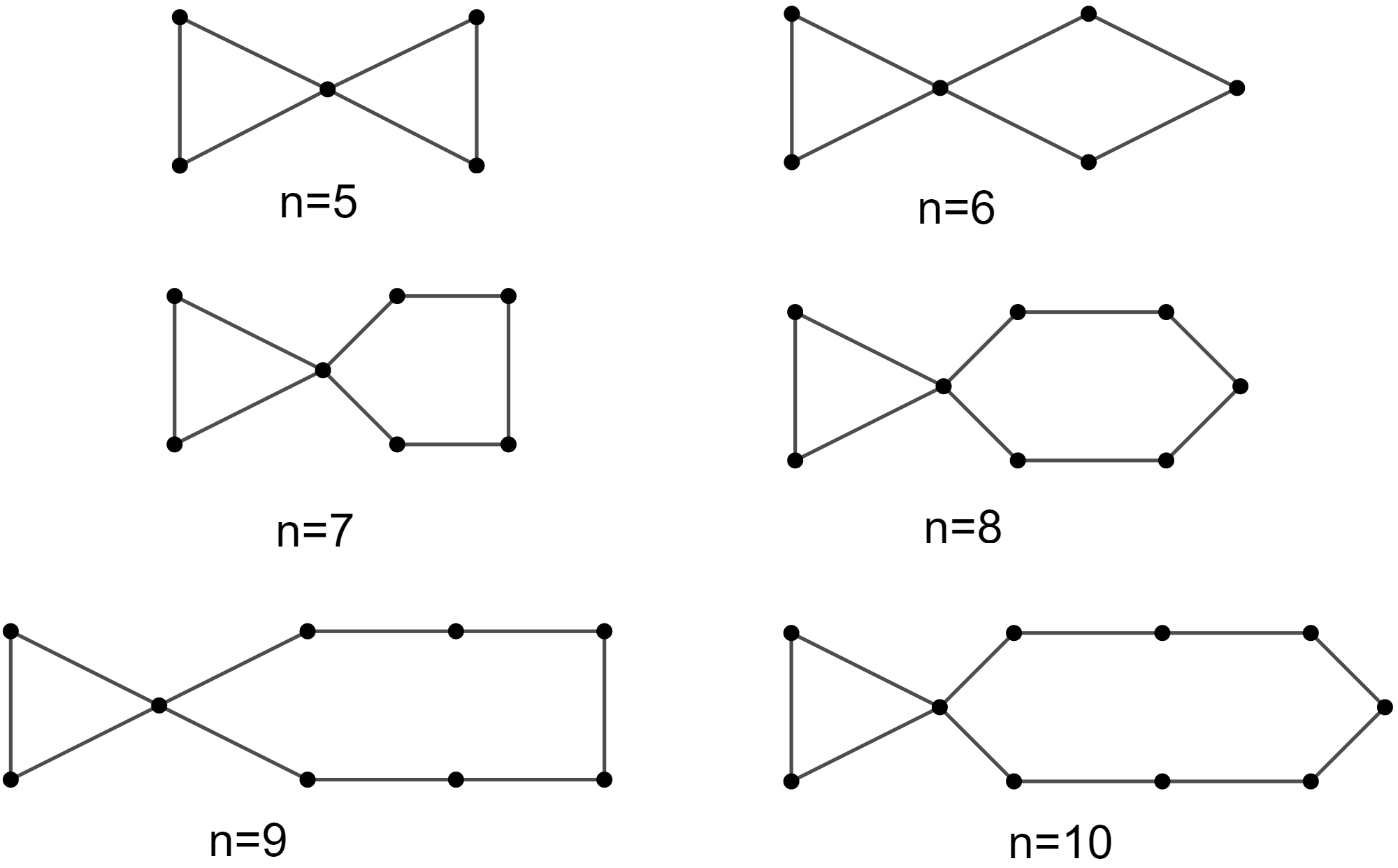}
    \caption{Graph of $B_1(n)$ family with minimal value of the $ABC_{GG}$ index for  $5 \leq n \leq 10$.}
    \label{B1_n10f}
\end{figure}

%

\section{Conclusion}

We finish this paper by presenting two conjectures related to the $ABC_{GG}$ index for any bicyclic graph. The following conjectures were motivated by computational experiments for all bicyclic graphs up to 16 vertices. The computational routines in Python are freely available at \mbox{https://github.com/20445/ProjetoTeste/blob/master/README.md.
}

Let $\mathcal{B'}_n$ be the family of all bicyclic graphs on $n$ vertices. The next conjecture states a lower bounds to the  $ABC_{GG}$ index among all graphs in $\mathcal{B'}_n.$ It is worth mentioning that the extremal graphs belong to the family $\mathcal{B}_{n},$ that is, the bicylic graphs with no pendant vertices. This fact makes the study of all graphs in $B_1(n), B_2(n),$ and $B_3(n)$ useful to prove the general case.

\begin{conjecture}
\label{conj2}
Let $G \in \mathcal{B}^{\prime}_n$ be a bicyclic graph of order $n \geq 9$. If $n$ is odd, then 
$$ABC_{GG}(G) \geq 2(n+1)\sqrt{\frac{n-2}{n^2-1}}.$$
If $n$ is even, then 
$$
ABC_{GG}(G) \geq \frac{6}{n}\sqrt{n-2} + 2(n-2)\sqrt{\frac{1}{n+2}}. 
$$
For $n$ odd, equality holds if and only if $G \cong B_{3}(4,2,n-1)$. For $n$ even, equality holds if and only if $G \cong B_{3}(6,3,n-2).$
\end{conjecture}

 Figure \ref{fig11} displays the extremal graphs of Conjecture \ref{conj2}  according to the partity of $n.$ 

\begin{figure}[h]
    \centering
    \includegraphics[height=4.0cm]{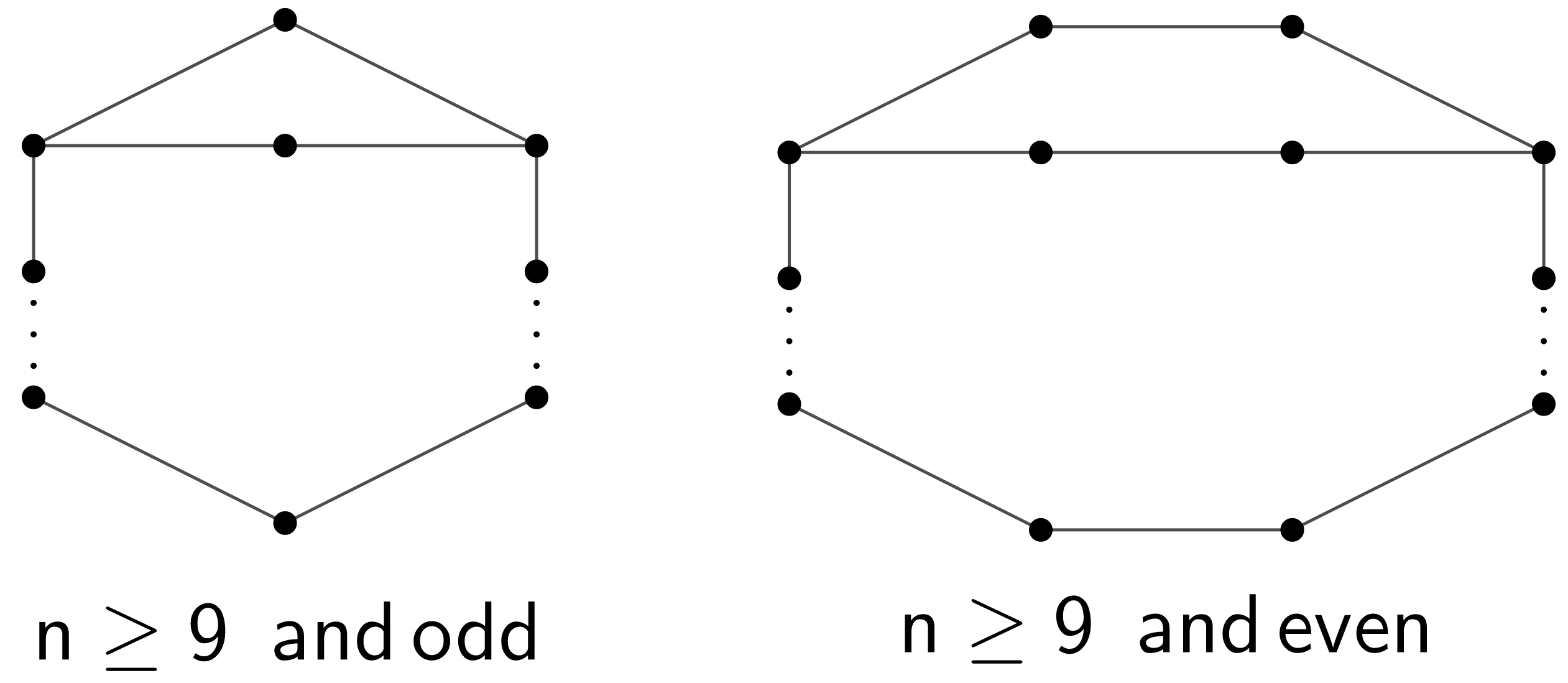}
    \caption{Bicyclic graphs with minimal value of $ABC_{GG}$ index for $n\geq 9$.}
    \label{fig11}
\end{figure}

Next, we present a conjecture about the upper bound to the $ABC_{GG}$ index for all bicyclic graphs. Let  $H$ be the graph obtained by adding $n-4$ pendant vertices to one  vertex of degree $3$ of the  complete graph a $K_4$ minus an edge. Figure \ref{fig12} displays the graphs with maximal $ABC_{GG}$ index for $n \geq 4,$ and the graph $H$ is the last graph for $n\geq 8.$

\begin{figure}[h]
    \centering
    \includegraphics[height=6.5cm]{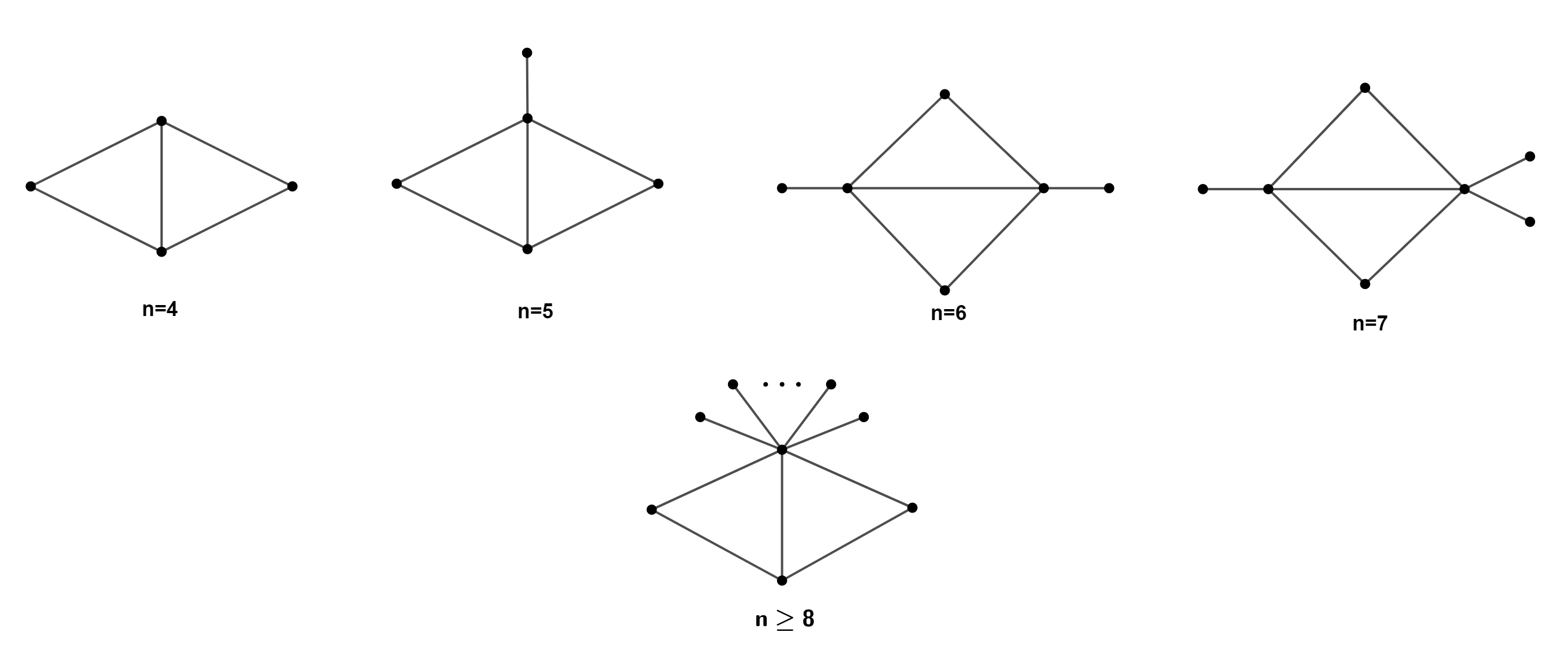}
    \caption{Bicyclic graphs with maximal value of $ABC_{GG}$ index for $n \geq 4.$}
    \label{fig12}
\end{figure}

\begin{conjecture}
\label{conj3}
Let $G \in \mathcal{B'}_n$ a bicyclic graph with order $n\geq 8$. Then, 

$$ABC_{GG}(G) \leq
(n-4)\sqrt{\frac{n-2}{n-1}}+\sqrt{\frac{n-4}{n-3}}+  
2\sqrt{\frac{n-3}{n-2}}+\frac{\sqrt{2}}{2}.$$
Equality holds if and only if $G$ is isomorphic to $H.$
\end{conjecture}

\end{document}